\documentclass[a4paper,12pt,reqno]{amsart}
\usepackage{amsfonts, color}
\usepackage{amsmath}
\usepackage{amssymb}
\usepackage[a4paper]{geometry}
\usepackage{mathrsfs}
\usepackage[colorlinks]{hyperref}
\usepackage{enumitem}
\usepackage{hyperref}

\numberwithin{equation}{section}
\theoremstyle{plain}
\newtheorem{theorem}{Theorem}[section]
\newtheorem{prop}[theorem]{Proposition}

\theoremstyle{definition}

\newtheorem{rem}[theorem]{Remark}

\def\X{\mathbb{X}}
\title[On fractional inequalities on metric measure spaces]{On fractional inequalities on metric measure spaces with polar decomposition}

\author[A. Kassymov]{Aidyn Kassymov}
\address{
  Aidyn Kassymov:
  \endgraf
  Institute of Mathematics and Mathematical Modeling
  \endgraf
  28 Shevchenko str.
  \endgraf
  050010 Almaty
  \endgraf
  Kazakhstan
  \endgraf
	{\it E-mail address}  {\rm kassymov@math.kz}}

\author[M. Ruzhansky]{Michael Ruzhansky}
\address{
  Michael Ruzhansky:
  \endgraf
  Department of Mathematics: Analysis, Logic and Discrete Mathematics
  \endgraf
  Ghent University, Belgium
  \endgraf
 and
  \endgraf
  School of Mathematical Sciences
  \endgraf
  Queen Mary University of London
  \endgraf
  United Kingdom
  \endgraf
  {\it E-mail address} {\rm michael.ruzhansky@ugent.be}
  }
\author[G. Zaur]{Gulnur Zaur}
\address{
Gulnur Zaur:
  \endgraf
  Institute of Mathematics and Mathematical Modeling
  \endgraf
  28 Shevchenko str.
  \endgraf
  050010 Almaty
  \endgraf
  Kazakhstan
  \endgraf
  and 
  \endgraf
  Al-Farabi Kazakh National University
  \endgraf
  71 Al-Farabi ave.
  \endgraf
  050040 Almaty
  \endgraf
  Kazakhstan
  \endgraf
	{\it E-mail address}  {\rm z.gulnur.t@gmail.com}
		}
	
\begin{document}
\thanks{AK and MR are supported by the FWO Odysseus 1 grant G.0H94.18N: Analysis and Partial Differential Equations and by the Methusalem programme of the Ghent University Special Research Fund (BOF) (Grant number 01M01021). MR is also supported by EPSRC grants EP/R003025/2 and EP/V005529/1. AK and GZ are supported by  the Science Committee of the Ministry of Science and Higher Education of the Republic of  Kazakhstan (Grant No. AP23484106)}
\maketitle
\begin{abstract}
    In this paper, we prove the fractional Hardy inequality on polarisable metric measure spaces. The integral Hardy inequality for $1<p\leq q<\infty$ is playing a key role in the proof. Moreover, we  also prove the fractional Hardy-Sobolev type inequality on metric measure spaces.   In addition, logarithmic Hardy-Sobolev and fractional Nash type inequalities on metric measure spaces are presented. In addition, we present applications on homogeneous groups and on the Heisenberg group. 
\end{abstract}

\section{Introduction}

In one of the famous works of G. H. Hardy (see \cite{Hardy}),  Hardy showed the following inequality:
\begin{equation}\label{one dim Hardy ineq}
    \int_a^{\infty}\frac{1}{x^p}\left(\int_{a}^xf(t)\,dt\right)^p\,dx\leq\left(\frac{p}{p-1}\right)^p\int_a^{\infty}f^p(x)\,dx,
\end{equation}
where $f\geq0$, $p>1$ and $a>0$. This inequality is called the direct integral Hardy inequality. This result has been extensively studied in various works, including \cite{Davies, Edmunds, Gogatishvili, Kufner, KufnerSamko, Opic}, and generalizations to higher dimensions have been established in \cite{Drabek}. The extension of the direct Hardy inequality to metric measure spaces for $1<p\leq q<\infty$ and $0<q<p$ with  $1<p<\infty$ were established in \cite{Verma} and \cite{Verma1}, respectively. Moreover, the reverse integral Hardy inequality for $q<0$ with $p\in (0,1)$ and $q\leq p<0$ were obtained in \cite{KRS19} and \cite{KRS21}.

In \cite{Dyda} and \cite{Frank}, the fractional Hardy inequality on $\mathbb{R}^n$ was studied, yielding the following inequality:
\begin{equation}\label{frac Hardy ineq on Rn}
    \int_{\mathbb{R}^n}\frac{|u(x)|^p}{|x|^{sp}_E}\, dx\leq C \int_{\mathbb{R}^n}\int_{\mathbb{R}^n}\frac{|u(x)-u(y)|^p}{|x-y|_E^{n+sp}}\, dxdy,
\end{equation}
where $s\in(0,1)$, $p>1$ is such that $n>sp$, and $|\cdot|_E$ denotes the Euclidean distance. Also, \eqref{frac Hardy ineq on Rn} was obtained for the case $n<sp$ in \cite{HKP97}. Also, by using the measure of the Euclidean ball $|B(x,r)|=|\sigma_{n-1}|r^{n}$, where $|\sigma_{n-1}|$ is the volume of the $n-1$ dimensional sphere, we can rewrite \eqref{frac Hardy ineq on Rn} as 
\begin{equation*}
     \int_{\mathbb{R}^n}\frac{|u(x)|^p}{|x|^{sp}_E}\, dx\leq C \int_{\mathbb{R}^n}\int_{\mathbb{R}^n}\frac{|u(x)-u(y)|^p}{|x-y|_E^{sp}|B(0,|x-y|_{E})|} dxdy.
\end{equation*}
Let  $(\mathbb{X},d)$ be a metric measure space. The measure $|\cdot|$ is doubling if there is a constant $C>0$, such that
\begin{equation*}
    |B(x, 2r)|\leq C|B(x, r)|, \quad x\in \mathbb{X}, \,\, r>0,
\end{equation*}
and $|\cdot|$ is reverse doubling, if there are constants $0<k<1$ and $0<C<1$ such that 
\begin{equation*}
    |B(x, kr)|\leq C|B(x, r)|, \quad x\in \mathbb{X}, \,\, 0<r<\text{diam}(\mathbb{X})/2,
\end{equation*}
where $\text{diam}(\mathbb{X})$ is the  diameter of $\mathbb{X}$.

For $E\subset\mathbb{X}$ and $r>0$, the open $r$-neighborhood of $E$ is the following set
\begin{equation*}
    E_r:=\{x\in\mathbb{X}: d(x, E)<r\}. 
\end{equation*}
The upper Assouad codimension of $E$ is denoted by $\overline{\text{co\,}\dim_{A}}(E)$, which is the infimum of all $D\geq 0$ for which there is a constant $C>0$ such that
\begin{equation*}
    \frac{|E_r\cap B(x, R)|}{|B(x, R)|}\geq C\left(\frac{r}{R}\right)^D
\end{equation*}
for all $x\in E$ and $0<r<R<\text{diam}(E)$. 
Hence,  in \cite{DLV22}, the authors obtained the localised fractional Hardy inequality on metric measure spaces with  doubling measures in the following form:
\begin{theorem}[Theorem 5.2, \cite{DLV22}]
 Let $0<s<1$ and $p>1$. Assume that the metric measure space $(\X, d)$ is complete and that $|\cdot|$ is reverse doubling. Let $E\subset \X$ be a closed set with $\overline{\rm{co\,}\dim_{A}}(E)<sp$ and let $u:\X\rightarrow \mathbb{R}$ be a continuous function such that $u=0$ on $E$. Then there exists $C>0$, independent of $u$, such that
 \begin{equation*}
     \int_{B\setminus E}\frac{|u(x)|^{p}}{d^{sp}(x,E)}dx\leq \int_{3B}\int_{3B}\frac{|u(x)-u(y)|^p}{d^{sp}(x,y)|B(x,d(x,y))|} dxdy,
 \end{equation*}
 whenever $B = B(w, r)$ with $w \in  E$ and $0 < r < \rm{diam}(E)$.
\end{theorem}
On stratified group, Ciatti, Cowling and Ricci proved the fractional Hardy inequality with fractional powers of the sub-Laplacian by using the boundedness of $T_{s}:u\mapsto |\cdot|^{-s}\mathcal{L}^{-\frac{s}{2}}u$ on the Lebesgue space,  where $|\cdot|$ is a homogeneous norm, $0<s<\frac{Q}{p}$ ($Q$ homogeneous dimension and $p>1$) and $\mathcal{L}$ is the sub-Laplacian. On the Heisenberg group, Thangavelu and Roncal (in \cite{Roncal})  obtained the fractional Hardy inequality in the case $p=2$ by using some explicit integral representations for the fractional powers of the conformal invariant fractional sub-Laplacian in the form $\mathcal{L}_{s}=|2T|^s\frac{\Gamma\left(\frac{\mathcal{L}}{2|T|}+\frac{1+s}{2}\right)}{\Gamma\left(\frac{\mathcal{L}}{2|T|}+\frac{1-s}{2}\right)}$. Furthermore, the authors derived this representation via the heat semigroup  generated by $\mathcal{L}_{s}$. Generally, for the fractional sub-Laplacian $\mathcal{L}^{s}$, the fractional Hardy inequality was obtained by establishing a relation between $\mathcal{L}_{s}$ and $\mathcal{L}^{s}$ using the expressions $U_{s}=\mathcal{L}_{s}\mathcal{L}^{-s}$ and $V_{s}=\mathcal{L}^{-1}_{1-s}\mathcal{L}\mathcal{L}^{-s}$. These operators are bounded on $L^{2}(\mathbb{H}^{n})$ (see \cite[Subsection 5.3]{Roncal}). In \cite{Adimurthi}, by using this integral representation, the authors  generalised a version of fractional Hardy inequality in the case $Q>sp$. Further investigations have considered fractional Hardy inequalities in different settings, such as homogeneous Lie groups. In addition, in the case of $Q>sp$ a fractional version of the Hardy inequality was obtained in \cite{Kassymov1}. The main aim of this paper is to obtain the fractional Hardy inequality on polarisable metric measure spaces. 

In \cite{Ghoussoub}, the integer order Hardy-Sobolev inequality was established, given by
\begin{equation*}
    \left(\int_{\mathbb{R}^n}\frac{|u(x)|^{p^*_{\beta}}}{|x|_{E}^{\beta}}\,dx\right)^{\frac{p}{p^*_{\beta}}}\leq C\int_{\mathbb{R}^n}|\nabla u(x)|^p\,dx,
\end{equation*}
where, $p^*_{\beta}=\frac{(n-\beta)p}{n-p}$, $1<p<n$, $\beta\in(0, p)$, and $\nabla$ is the Euclidean gradient. The Sobolev inequality is obtained when $\beta=0$, and the classical Hardy inequality arises when $\beta=p$. The fractional version of the Hardy-Sobolev inequality was introduced in \cite{Yang} as follows:
\begin{equation*}
    \left(\int_{\mathbb{R}^n}\frac{|u(x)|^{p^*_{s,\beta}}}{|x|_{E}^{\beta}}\,dx\right)^{\frac{p}{p^*_{s,\beta}}}\leq C\int_{\mathbb{R}^n}\int_{\mathbb{R}^n}\frac{|u(x)-u(y)|^p}{|x-y|^{n+sp}}\,dxdy,
\end{equation*}
where, $p^*_{s,\beta}=\frac{(n-\beta)p}{n-sp}$, $0<\beta<sp<n$, and $s\in(0, 1)$. Similarly, we get the fractional Sobolev inequality if $\beta=0$ and the fractional Hardy inequality if $\beta=p$. We refer to \cite{Avkhadiev, Benguria, Dyda2, Frank2} for this topic.
 
In this paper, we extend the fractional Hardy and fractional Hardy-Sobolev type inequalities to metric measure spaces. We consider a metric space $\mathbb{X}$ equipped with a Borel measure $dx$ allowing for the polar decomposition at $a\in \mathbb{X}$. This decomposition involves the existence of a locally integrable function $\lambda\in L^1_{loc}$, such that for any $f\in L^1(\mathbb{X})$ we have
\begin{equation}\label{polar decomposition}
    \int_{\mathbb{X}}f(x)\,dx=\int_{0}^{\infty}\int_{\Sigma_r}f(r, \omega)\lambda(r, \omega)\,d\omega_r dr,
\end{equation}
for the set $\Sigma_r=\{x\in\mathbb{X}: d(x,a)=r\}\subset \mathbb{X}$ with a measure on it denoted by $d\omega = d\omega_r$.

The condition \eqref{polar decomposition} can be obtained as the standard polar decomposition formula, if a differentiable structure exists on $\mathbb{X}$. In general, when $\mathbb{X}$ does not have a differentiable structure, we can not obtain the function $\lambda(r, \omega)$ as the Jacobian of a polar change of coordinates.  Let us consider some examples of $\mathbb{X}$ for which the condition \eqref{polar decomposition} holds, and have the different expressions for $\lambda(r, \omega)$:

\begin{enumerate}[label=(\Roman*)]
    \item Euclidean space $\mathbb{R}^n: \lambda (r,\omega)=r^{n-1}$.
    \item Homogeneous groups $\mathbb{G}$: $\lambda(r, \omega) = r^{Q-1}$, where $Q$ represents the homogeneous dimension of the
group. Homogeneous groups have been extensively studied by Folland and Stein \cite{Folland}, see also \cite{Fischer} and \cite{RuzhanskySuragan}.
\item Hyperbolic spaces $\mathbb{H}^n:\lambda(r,\omega)=(\sinh r)^{n-1}$.
\item Cartan–Hadamard manifolds: Here, we consider a Riemannian manifold $(M, g)$ with sectional curvature denoted as $K_M$. A Riemannian manifold is called a Cartan–Hadamard manifold, if it possesses the following characteristics: it is complete, simply connected, and has non-positive sectional curvature, meaning that the sectional curvature $K_M \leq 0$ for each plane section at every point of $M$. To proceed, we fix a point $a \in M$ and denote by $\rho(x) = d(x, a)$
the geodesic distance from $x$ to $a$ on $M$. The exponential map $\exp _a : T_aM \rightarrow M$ is a diffeomorphism, see e.g. Helgason \cite{Helgason}. Let $J(\rho, \omega)$ be the density function on $M$, see e.g. \cite{Gallot}. Consequently, the polar decomposition
takes the following form:
\begin{equation*}
    \int_{M}f(x)dx=\int_{0}^{\infty}\int_{\mathbb{S}^{n-1}}f(\exp _a(\rho\omega))J(\rho, \omega)\rho^{n-1}\, d\rho d\omega,
\end{equation*}
so that we have \eqref{polar decomposition} with $\lambda(\rho, \omega) = J(\rho, \omega)\rho^{n-1}$.
\end{enumerate}

Conditions for \eqref{polar decomposition} to hold on general metric measure spaces were recently analysed in \cite{Avetisyan}.

\section{Main results}

Consider $\mathbb{X}$ as a metric measure space. For brevity, we denote a ball in $\mathbb{X}$ with center $a$ and radius $r$ as $B(a, r)$, i.e.,
\begin{equation*}
    B(a, r):=\left\{x\in\mathbb{X}: d(x,a)<r\right\},
\end{equation*}
where $d$ is the metric on $\mathbb{X}$. Let $a$ be a fixed point in $\mathbb{X}$, and we will use the notation $|x|_a = d(a, x)$.

Furthermore, let us define $|B(a, r)|$ as the measure of the ball $B(a, r)$.

Now, turning our attention to this setting, we recall the integral Hardy inequality on metric measure spaces.

\begin{theorem}[\cite{Verma} Integral Hardy inequality]\label{2.1}
    Let $1<p\leq q<\infty$. And let $\mathbb{X}$ be a metric measure space with a polar decomposition at $a$. Let $g, h$ be positive measurable functions on $\mathbb{X}$ such that $g\in L^1(\mathbb{X}\backslash \{0\})$ and $h^{1-p'}\in L^1_{loc}(\mathbb{X})$. Then the following inequality 
    \begin{equation}\label{integralHardyinequalityonmetricmeasurespaces}
        \left(\int_{\mathbb{X}}\left(\int_{B(a, |x|_a)}\left|f(y)\right|\,dy\right)^q g(x)\, dx\right)^{\frac{1}{q}}\leq C_H \left(\int_{\mathbb{X}}\left|f(x)\right|^p h(x)\, dx\right)^{\frac{1}{p}}
    \end{equation}
    holds for all measurable functions $f: \mathbb{X}\rightarrow \mathbb{C}$ if and only if the following condition holds:
    \begin{equation}\label{Condition of the int Hardy ineq}
        D_1:=\sup_{x\neq 0}\left[\int_{\mathbb{X}\backslash B(a, |x|_a)}g(y)\, dy\right]^{\frac{1}{q}}\left[\int_{B(a, |x|_a)}h^{1-p'}(y)\, dy\right]^{\frac{1}{p'}}<\infty.
    \end{equation}
    Moreover,
    \begin{equation*}
        D_1\leq C_H\leq (p')^{\frac{1}{p'}}p^{\frac{1}{q}}D_1, 
    \end{equation*}
    where $\frac{1}{p}+\frac{1}{p'}=1$.
\end{theorem}

\subsection{Fractional Hardy inequality}
As the main result of this section, we present the fractional Hardy inequality on metric measure spaces with polar decomposition.
\begin{theorem}\label{frac Hardy ineq}
Let $\mathbb{X}$ be a metric measure space with polar decomposition at $a$, and let $s>0$ and $p>1$. Assume that $v(x,y)$ is a positive measurable function on $\mathbb{X}\times\mathbb{X}$ and denote 
\begin{equation}\label{1}
    A(x):=\left(\frac{1}{\left|B(a, |x|_a)\right|}\int_{B(a, |x|_a)}v^{1-p'}(x,y)dy\right)^{1-p}.
\end{equation}
Suppose that $\left(p'\right)^{\frac{1}{p'}}p^{\frac{1}{p}}D_1<1$, where 
\begin{equation}\label{2}
    D_1=\sup_{x\neq a}\left[\left(\int_{\mathbb{X}\backslash B(a, |x|_a)}\frac{A(y)}{|y|_a^{sp}|B(a,|y|_a)|^p}\, dy\right)^{\frac{1}{p}}\left(\int_{B(a, |x|_a)}\left(\frac{A(y)}{|y|_a^{sp}}\right)^{1-p'}\, dy\right)^{\frac{1}{p'}}\right]<\infty.
\end{equation}
Then the following inequality holds
\begin{equation}\label{Hardy}
    \int_{\mathbb{X}}\frac{A(x)\left|u(x)\right|^p}{|x|_a^{sp}}\, dx\leq \frac{2^{sp}}{(1-\left(p'\right)^{\frac{1}{p'}}p^{\frac{1}{p}}D_1)^p }\int_{\mathbb{X}}\int_{\mathbb{X}}\frac{\left|u(x)-u(y)\right|^p v(x,y)}{d^{sp}(x,y)\left|B\left(a, \frac{d(x,y)}{2}\right)\right|}dx\, dy.
\end{equation}
\end{theorem}
\begin{proof}
By Minkowski inequality, we have
\begin{equation}\label{4}
\begin{split}
     \left(\int_{\mathbb{X}}\frac{A(x)\left|u(x)\right|^p}{|x|_a^{sp}}\, dx\right)^{\frac{1}{p}}&=\left(\int_{\mathbb{X}}\frac{A(x)}{|x|_a^{sp}}\left|u(x)-\frac{1}{\left|B(a, |x|_a)\right|}\int_{B(a, |x|_a)}u(y)\, dy\right. \right. \\
     &+\left.\frac{1}{\left|B(a, |x|_a)\right|}\left.\int_{B(a, |x|_a)} u(y)\, dy\right|^p\, dx\right)^{\frac{1}{p}}\\
     &\leq\left(\int_{\mathbb{X}}\frac{A(x)}{|x|_a^{sp}}\left|u(x)-\frac{1}{\left|B(a, |x|_a)\right|}\int_{B(a, |x|_a)}u(y)\,dy\right|^p dx\right)^{\frac{1}{p}}\\
     &+\left(\int_{\mathbb{X}}\frac{A(x)}{|x|_a^{sp}\left|B(a, |x|_a)\right|^p}\left|\int_{B(a, |x|_a)} u(y)\, dy\right|^p\, dx\right)^{\frac{1}{p}}=I_1+I_2,
\end{split}
\end{equation}
    where 
    \begin{equation*}
        I_1=\left(\int_{\mathbb{X}}\frac{1}{|x|_a^{sp}}\left|u(x)-\frac{1}{B(a, |x|_a)}\int_{B(a, |x|_a)}u(y)\,dy\right|^p dx\right)^{\frac{1}{p}}
    \end{equation*}
and 
\begin{equation*}
    I_2=\left(\int_{\mathbb{X}}\frac{A(x)}{|x|_a^{sp}\left|B(a, |x|_a)\right|^p}\left|\int_{B(a, |x|_a)} u(y)\, dy\right|^p\, dx\right)^{\frac{1}{p}}.
\end{equation*}
By H\"{o}lder inequality with $\frac{1}{p}+\frac{1}{p'}=1$, we have
\begin{equation}\label{7}
    \begin{split}
    &\left|\int_{B\left(a, |x|_a\right)}\left(u(x)-u(y)\right) v(x,y)^{\frac{1}{p}} v(x,y)^{-\frac{1}{p}}\, dy\right|^p\\
    &\leq\left(\int_{B\left(a, |x|_a\right)}\left|u(x)-u(y)\right|^p v(x,y)\, dy\right)\left(\int_{B\left(a, |x|_a\right)}v^{1-p'}(x,y)\, dy\right)^{p-1}.
    \end{split}
\end{equation}

For $0 < r_1 \leq r_2$, we have $\left|B(a, r_1)\right| \leq \left|B(a, r_2)\right|$.

Furthermore, using the property that $|y|_a \leq |x|_a$ for all $y \in B\left(a, |x|_a\right)$ together with the triangle inequality, we obtain
\begin{equation}\label{8}
d(x,y)\leq |x|_a+|y|_a\leq 2|x|_a,
\end{equation}
and hence
\begin{equation}\label{9}
\left|B\left(a, \frac{d(x,y)}{2}\right)\right|\leq\left|B(a,|x|_a)\right|.
\end{equation}
Therefore, we get
\begin{equation}\label{10}
    \begin{split}
    I_1&=\left(\int_{\mathbb{X}}\frac{A(x)}{|x|_a^{sp}}\left|u(x)-\frac{1}{\left|B(a, |x|_a)\right|}\int_{B(a, |x|_a)}u(y)\,dy\right|^p dx\right)^{\frac{1}{p}}\\
    &=\left(\int_{\mathbb{X}}\frac{A(x)}{|x|_a^{sp}\left|B(a, |x|_a)\right|^p}\left|u(x)\left|B(a, |x|_a)\right|-\int_{B(a, |x|_a)}u(y)\,dy\right|^p dx\right)^{\frac{1}{p}}\\
    &=\left(\int_{\mathbb{X}}\frac{A(x)}{|x|_a^{sp}\left|B(a, |x|_a)\right|^p}\left|\int_{B(a, |x|_a)}\left(u(x)-u(y)\right)\,dy\right|^p dx\right)^{\frac{1}{p}}\\
    &\stackrel{(\ref{7})}{\leq}\left(\int_{\mathbb{X}}\frac{A(x)\left(\int_{B\left(a, |x|_a\right)}v^{1-p'}(x,y)\, dy\right)^{p-1}}{|x|_a^{sp}\left|B(a, |x|_a)\right|^{p-1}\left|B(a, |x|_a)\right|} \int_{B\left(a, |x|_a\right)}\left|u(x)-u(y)\right|^p v(x,y)\, dy   dx\right)^{\frac{1}{p}}\\
    &\stackrel{(\ref{1})}{=} \left(\int_{\mathbb{X}}\frac{1}{|x|_a^{sp}\left|B(a, |x|_a)\right|} \int_{B\left(a, |x|_a\right)}\left|u(x)-u(y)\right|^p v(x,y)\, dy   dx\right)^{\frac{1}{p}}  \\
    &\stackrel{(\ref{8}), (\ref{9})}{\leq} 2^{s}\left(\int_{\mathbb{X}}\int_{B\left(a, |x|_a\right)}\frac{\left|u(x)-u(y)\right|^p v(x,y)}{d^{sp}(x,y)\left|B\left(a, \frac{d(x,y)}{2}\right)\right|}\, dydx\right)^{\frac{1}{p}}\\
    &\leq2^{s}\left(\int_{\mathbb{X}}\int_{\mathbb{X}}\frac{\left|u(x)-u(y)\right|^p v(x,y)}{d^{sp}(x,y)\left|B\left(a, \frac{d(x,y)}{2}\right)\right|}dxdy\right)^{\frac{1}{p}}.
     \end{split}
\end{equation}

    By assumption \eqref{2}, we can observe that the integral Hardy inequality (Theorem \ref{2.1}) applies to functions $g(x) = \frac{A(x)}{|x|_a^{sp}\left|B(a, |a|_a)\right|}$ and $h(x) = \frac{A(x)}{|x|_a^{sp}}$. This yields
    \begin{equation}\label{11}
        \begin{split}
        I_2&=\left(\int_{\mathbb{X}}\frac{A(x)}{|x|_a^{sp}\left|B(a, |x|_a)\right|^p}\left|\int_{B(a, |x|_a)} u(y)\, dy\right|^p\, dx\right)^{\frac{1}{p}}\leq C_H \left(\int_{\mathbb{X}}\frac{A(x)\left|u(x)\right|^p}{|x|_a^{sp}}\,dx\right)^{\frac{1}{p}}\\
        &\leq \left(p'\right)^{\frac{1}{p'}}p^{\frac{1}{p}}D_1 \left(\int_{\mathbb{X}}\frac{A(x)\left|u(x)\right|^p}{|x|_a^{sp}}\,dx\right)^{\frac{1}{p}}.
    \end{split}
    \end{equation}

    Substituting \eqref{10} and \eqref{11} into \eqref{4}, and using that $\left(p'\right)^{\frac{1}{p'}}p^{\frac{1}{p}}D_1<1$, we obtain
    \begin{equation*}
    \left(\int_{\mathbb{X}}\frac{A(x)\left|u(x)\right|^p}{|x|_a^{sp}}\,dx\right)^{\frac{1}{p}}\leq \frac{2^{s}}{1-\left(p'\right)^{\frac{1}{p'}}p^{\frac{1}{p}}D_1 }\left(\int_{\mathbb{X}}\int_{\mathbb{X}}\frac{\left|u(x)-u(y)\right|^p v(x,y)}{d^{sp}(x,y)\left|B\left(a, \frac{d(x,y)}{2}\right)\right|}dxdy\right)^{\frac{1}{p}},
    \end{equation*}
    which completes the proof.
\end{proof}

\subsection{Fractional Hardy-Sobolev type inequality} In this subsection, we show the fractional Hardy-Sobolev inequality on metric measure spaces with polar decomposition. 
\begin{theorem}\label{theo 3.2}
    Let $\mathbb{X}$ be a metric measure space with a polar decomposition at $a$, $s>0$, and $1<p,q<\infty$. Assume that $z(x)$ and $v(x)$ are positive measurable functions on $\mathbb{X}$ and denote 
    \begin{equation}\label{12}
        V(x)=\int_{B(a,|x|_a)} v(y)\, dy,
    \end{equation}
    \begin{equation}\label{13}
        A(x):=\frac{1}{\left|B(a, |x|_a)\right|^{\frac{q}{p}}}\left(\int_{B(a,|x|_a)}\left(\frac{v^p(y)}{z(y)}\right)^{\frac{1}{p-1}}\, dy\right)^{-\frac{q}{p'}}V^q (x) v(x).
    \end{equation}
     Suppose that $\left(q'\right)^{\frac{1}{q'}}q^{\frac{1}{q}}D_1<1$, where 
    \begin{equation}\label{14}
        D_1=\sup_{x\neq0}\left[\left(\int_{\mathbb{X}\backslash B(a, |x|_a)}\frac{A(y)}{|y|_a^{sq}V^q(y)}\, dy\right)^{\frac{1}{q}}\left(\int_{B(a, |x|_a)}\left(\frac{A(y)v^{-q} (y)}{|y|_a^{sq}}\right)^{(1-q')}\, dy\right)^{\frac{1}{q'}}\right]<\infty.
        \end{equation}
    Then we have
    \begin{equation}\label{frac Sovolev}
    \begin{split}
        &\int_{\mathbb{X}}\frac{\left|u(x)\right|^q A(x)}{|x|_a^{sq}}\, dx\\
        &\leq \frac{2^{sq}}{(1-\left(q'\right)^{\frac{1}{q'}}q^{\frac{1}{q}}D_1)^q }\int_{\mathbb{X}}v(y)\left(\int_{\mathbb{X}}\frac{\left|u(x)-u(y)\right|^p z(x)}{d^{sp}(x,y)\left|B\left(a, \frac{d(x,y)}{2}\right)\right|}dx\right)^{\frac{q}{p}} \, dy.
    \end{split}
    \end{equation}
\end{theorem}
    
    \begin{proof} We follow the same pattern as in the proof of Theorem \ref{frac Hardy ineq}.
        By Minkowski inequality, we get

        \begin{equation}\label{16}
            \begin{split}
                \left(\int_{\mathbb{X}}\frac{\left|u(y)\right|^q A(y)}{|y|_a^{sq}}\, dy\right)^{\frac{1}{q}}&=\left(\int_{\mathbb{X}}\frac{A(y)}{|y|_a^{sq}}\left|u(y)-\frac{1}{V(y)}\int_{B(a, |y|_a)}u(x)v(x)\, dx\right.\right.\\
                &+\left.\left.\frac{1}{V(y)}\int_{B(a, |y|_a)}u(x)v(x)\, dx\right|^q\, dy\right)^{\frac{1}{q}}\\
                &\leq \left(\int_{\mathbb{X}}\frac{ A(y)}{|y|_a^{sq}}\left|u(y)-\frac{1}{V(y)}\int_{B(a, |y|_a)}u(x)v(x)\, dx\right|^q\, dy\right)^{\frac{1}{q}}\\
                &+\left(\int_{\mathbb{X}}\frac{A(y)}{|y|_a^{sq}V^q(y)}\left|\int_{B(a, |y|_a)}u(x)v(x)\, dx\right|^q\, dy\right)^{\frac{1}{q}}=I_1+I_2,
            \end{split}
        \end{equation}
        where
\begin{equation*}
    I_1=\left(\int_{\mathbb{X}}\frac{ A(y)}{|y|_a^{sq}}\left|u(y)-\frac{1}{V(y)}\int_{B(a, |y|_a)}u(x)v(x)\, dx\right|^q\, dy\right)^{\frac{1}{q}},
\end{equation*}
       and
       \begin{equation*}
           I_2=\left(\int_{\mathbb{X}}\frac{A(y)}{|y|_a^{sq}V^q(y)}\left|\int_{B(a, |y|_a)}u(x)v(x)\, dx\right|^q\, dy\right)^{\frac{1}{q}}.
       \end{equation*}
By H\"{o}lder inequality with $\frac{1}{p}+\frac{1}{p'}=1$, we get
\begin{equation}\label{19}
    \begin{split}
        &\left|\int_{B(a, |y|_a)}\left(u(y)-u(x)\right)v(x)\, dx\right|^q\\
        &\leq \left(\int_{B(a, |y|_a)}\left|u(x)-u(y)\right|^p z(x)\, dx\right)^{\frac{q}{p}}\left(\int_{B(a, |y|_a)}\left(\frac{v^p(x)}{z(x)}\right)^{\frac{1}{p-1}}\, dx\right)^{\frac{q}{p'}}.
    \end{split}
\end{equation}

By using \eqref{19} together with \eqref{8} and \eqref{9}, we have
\begin{equation}\label{22}
    \begin{split}
        I_1&=\left(\int_{\mathbb{X}}\frac{ A(y)}{|y|_a^{sq}}\left|u(y)-\frac{1}{V(y)}\int_{B(a, |y|_a)}u(x)v(x)\, dx\right|^q\, dy\right)^{\frac{1}{q}}\\
        &\stackrel{(\ref{12})}{=}\left(\int_{\mathbb{X}}\frac{ A(y)}{|y|_a^{sq}V^q(y)}\left|\int_{B(a, |y|_a)}\left(u(y)-u(x)\right)v(x)\, dx\right|^q\, dy\right)^{\frac{1}{q}}\\
        &\stackrel{(\ref{19})}{\leq}\left(\int_{\mathbb{X}}\frac{ A(y) \left(\int_{B(a, |y|_a)}\left|u(x)-u(y)\right|^p z(x)\, dx\right)^{\frac{q}{p}}}{|y|_a^{sq}V^q(y)\left(\int_{B(a, |y|_a)}\left(\frac{v^p(x)}{z(x)}\right)^{\frac{1}{p-1}}\, dx\right)^{-\frac{q}{p'}}} \, dy\right)^{\frac{1}{q}}\\
        &\stackrel{(\ref{13})}{=}\left(\int_{\mathbb{X}}\frac{ \left(\int_{B(a, |y|_a)}\left|u(x)-u(y)\right|^p z(x)\, dx\right)^{\frac{q}{p}}v(y)}{|y|_a^{sq}\left|B(a, |y|_a)\right|^{\frac{q}{p}}}  \, dy\right)^{\frac{1}{q}}\\
        &\stackrel{(\ref{8}), (\ref{9})}{\leq}2^s\left(\int_{\mathbb{X}}\left(\int_{B(a, |y|_a)}\frac{\left|u(x)-u(y)\right|^p z(x)}{d^{sp}(x,y)\left|B\left(a, \frac{d(x,y)}{2}\right)\right|}\, dx\right)^{\frac{q}{p}}v(y)\, dy\right)^{\frac{1}{q}}\\
        &\leq 2^s\left(\int_{\mathbb{X}}\left(\int_{\mathbb{X}}\frac{\left|u(x)-u(y)\right|^p z(x)}{d^{sp}(x,y)\left|B\left(a, \frac{d(x,y)}{2}\right)\right|}\, dx\right)^{\frac{q}{p}}v(y)\, dy\right)^{\frac{1}{q}}.
    \end{split}
\end{equation}

    With regard to assumption \eqref{14}, we see that the integral Hardy inequality (see Theorem \ref{2.1}) holds for the functions $g(x):=\frac{A(x)}{|x|_a^{sq}V^q(x)}$ and $h(x):=\frac{A(x)}{|x|_a^{sq}v^q(x)}$, and hence we have

    \begin{equation}\label{23}
        \begin{split}
            I_2&=\left(\int_{\mathbb{X}}\frac{A(y)}{|y|_a^{sq}(y,a)V^q(y)}\left|\int_{B(a, |y|_a)}u(x)v(x)\, dx\right|^q\, dy\right)^{\frac{1}{q}}\\
            &\leq C_H  \left(\int_{\mathbb{X}}\frac{\left|u(x)\right|^q A(x)}{|x|_a^{sq}}\, dx\right)^{\frac{1}{q}}\\
            &\leq \left(q'\right)^{\frac{1}{q'}}q^{\frac{1}{q}}D_1 \left(\int_{\mathbb{X}}\frac{\left|u(x)\right|^q A(x)}{|x|_a^{sq}}\, dx\right)^{\frac{1}{q}}.
        \end{split}
    \end{equation}

    Finally, combining \eqref{22}, \eqref{23}, and considering that $\left(q'\right)^{\frac{1}{q'}}q^{\frac{1}{q}}D_1<1$, we derive
    \begin{equation}
    \begin{split}
       &\left(\int_{\mathbb{X}}\frac{\left|u(x)\right|^q A(x)}{|x|_a^{sq}}\, dx\right)^{\frac{1}{q}}\\
       &\leq \frac{2^{s}}{1-\left(q'\right)^{\frac{1}{q'}}q^{\frac{1}{q}}D_1 }\left(\int_{\mathbb{X}}\left(\int_{\mathbb{X}}\frac{\left|u(x)-u(y)\right|^p z(x)}{d^{sp}(x,y)\left|B\left(a, \frac{d(x,y)}{2}\right)\right|}\, dx\right)^{\frac{q}{p}} v(y)\, dy\right)^{\frac{1}{q}}, 
    \end{split}
    \end{equation}
    completing the proof. 
    \end{proof}

\subsection{Logarithmic Hardy-Sobolev type inequality}
Let us recall logarithmic H\"{o}lder inequality on general measure spaces. 

\begin{theorem}[Logarithmic H\"{o}lder inequality \cite{chatzakou2021logarithmic}]\label{theo2.1}
  Let $\mathbb{X}$ be a measure space. Let $u\in L^p\left(\mathbb{X}\right)\cap L^q\left(\mathbb{X}\right)\backslash \{0\}$ with some $1 < p < q < \infty$. Then we have 
  \begin{equation}\label{logHolder}
      \int_{\mathbb{X}}\frac{\left|u\right|^p}{\left\|u\right\|^p_{L^p\left(\mathbb{X}\right)}}\log\left(\frac{\left|u\right|^p}{\left\|u\right\|^p_{L^p\left(\mathbb{X}\right)}}\right)\, dx\leq\frac{q}{q-p}\log\left(\frac{\left\|u\right\|^p_{L^q\left(\mathbb{X}\right)}}{\left\|u\right\|^p_{L^p\left(\mathbb{X}\right)}}\right).
  \end{equation}
\end{theorem}
Now we show the logarithmic fractional Hardy-Sobolev inequality on polarisable metric measure space.
\begin{theorem}\label{theo3.3}
    Let $\mathbb{X}$ be a metric measure space with a polar decomposition at $a$, and let $s>0$, $1<p<q<\infty$. Assume that $v(x)$ and $z(x)$ are non-negative measurable functions on $\mathbb{X}$. Suppose that $\left(q'\right)^{\frac{1}{q'}}q^{\frac{1}{q}}D_1<1$, where
    \begin{equation*}
    D_1=\sup_{x\neq0}\left[\left(\int_{\mathbb{X}\backslash B(a, |x|_a)}\frac{A(y)}{|y|_a^{sq}V^q(y)}\, dy\right)^{\frac{1}{q}}\left(\int_{B(a, |x|_a)}\left(\frac{A(y)v^{-q} (y)}{|y|_a^{sq}}\right)^{(1-q')}\, dy\right)^{\frac{1}{q'}}\right]<\infty,
    \end{equation*}
    \begin{equation*}
         V(x)=\int_{B(a,|x|_a)} v(y)\, dy,
    \end{equation*}
    and
    \begin{equation*}
         A(x):=\frac{1}{\left|B(a, |x|_a)\right|^{\frac{q}{p}}}\left(\int_{B(a,|x|_a)}\left(\frac{v^p(y)}{z(y)}\right)^{\frac{1}{p-1}}\, dy\right)^{-\frac{q}{p'}}V^q (x) v(x).
    \end{equation*}
    Under these conditions, we have the inequality
    \begin{equation}\label{logHardySovolev}
    \begin{split}
        &\int_{\mathbb{X}}\frac{\left(\frac{A^{\frac{1}{q}}|u|}{|\cdot|_a^s}\right)^p}{\left\|\frac{A^{\frac{1}{q}}u}{|\cdot|_a^s}\right\|^p_{L^p(\mathbb{X})}}\log\left(\frac{\left(\frac{A^{\frac{1}{q}}|u|}{|\cdot|_a^s}\right)^p}{\left\|\frac{A^{\frac{1}{q}}u}{|\cdot|_a^s}\right\|^p_{L^p(\mathbb{X})}}\right)\\
        &\leq\frac{q}{q-p}\log\left(C\frac{\left(\int_{\mathbb{X}}\left(\int_{\mathbb{X}}\frac{\left|u(x)-u(y)\right|^p z(x)}{d^{sp}(x,y)\left|B\left(a, \frac{d(x,y)}{2}\right)\right|}\,dx\right)^{\frac{q}{p}} v(y)\, dy\right)^{\frac{p}{q}}}{\left\|\frac{A^{\frac{1}{q}}u}{|\cdot|_a^s}\right\|^p_{L^p(\mathbb{X})}}\right), 
        \end{split}
    \end{equation}
    where C is a positive constant independent of $u$.
\end{theorem}
\begin{proof}
   By logarithmic H\"{o}lder inequality \eqref{logHolder} and fractional Hardy-Sobolev type inequality \eqref{frac Sovolev}, we get
   \begin{equation*}
   \begin{split}
       &\int_{\mathbb{X}}\frac{\left(\frac{A^{\frac{1}{q}}|u|}{|\cdot|_a^s}\right)^p}{\left\|\frac{A^{\frac{1}{q}}u}{|\cdot|_a^s}\right\|^p_{L^p(\mathbb{X})}}\log\left(\frac{\left(\frac{A^{\frac{1}{q}}|u|}{|\cdot|_a^s}\right)^p}{\left\|\frac{A^{\frac{1}{q}}u}{|\cdot|_a^s}\right\|^p_{L^p(\mathbb{X})}}\right)\leq\frac{q}{q-p}\log\left(\frac{\left\|\frac{A^{\frac{1}{q}}u}{|\cdot|_a^s}\right\|^p_{L^q(\mathbb{X})}}{\left\|\frac{A^{\frac{1}{q}}u}{|\cdot|_a^s}\right\|^p_{L^p(\mathbb{X})}}\right)\\
       &\leq\frac{q}{q-p}\log\left(C\frac{\left(\int_{\mathbb{X}}\left(\int_{\mathbb{X}}\frac{\left|u(x)-u(y)\right|^p z(x)}{d^{sp}(x,y)\left|B\left(a, \frac{d(x,y)}{2}\right)\right|}\,dx\right)^{\frac{q}{p}} v(y)\, dy\right)^{\frac{p}{q}}}{\left\|\frac{A^{\frac{1}{q}}u}{|\cdot|_a^s}\right\|^p_{L^p(\mathbb{X})}}\right),
       \end{split}
   \end{equation*}
   completing the proof. 
\end{proof}
\subsection{Fractional Nash type inequality}
Finally we show the fractional Nash type inequality on metric measure spaces with polar decomposition.
\begin{theorem}
    Let $\mathbb{X}$ be a metric measure space with a polar decomposition at $a$, and let $s>0$, $2<q<\infty$. Assume that $v(x)$ and $z(x)$ are non-negative measurable functions on $\mathbb{X}$. Suppose that $\left(q'\right)^{\frac{1}{q'}}q^{\frac{1}{q}}D_1<1$, where
    \begin{equation*}
    D_1=\sup_{x\neq0}\left[\left(\int_{\mathbb{X}\backslash B(a, |x|_a)}\frac{A(y)}{|y|_a^{sq}V^q(y)}\, dy\right)^{\frac{1}{q}}\left(\int_{B(a, |x|_a)}\left(\frac{A(y)v^{-q} (y)}{|y|_a^{sq}}\right)^{(1-q')}\, dy\right)^{\frac{1}{q'}}\right]<\infty,
    \end{equation*}
    \begin{equation*}
         V(x)=\int_{B(a,|x|_a)} v(y)\, dy,
    \end{equation*}
    and
    \begin{equation*}
         A(x):=\frac{1}{\left|B(a, |x|_a)\right|^{\frac{q}{2}}}\left(\int_{B(a,|x|_a)}\left(\frac{v^2(y)}{z(y)}\right)\, dy\right)^{-\frac{q}{2}}V^q(x) v(x).
    \end{equation*}
    Then the following inequality holds
\begin{equation*}
    \left\|\frac{A^{\frac{1}{q}}u}{|\cdot|_a}\right\|^{4-\frac{4}{q}}_{L^2(\mathbb{X})}\leq C \left(\int_{\mathbb{X}}\left(\int_{\mathbb{X}}\frac{\left|u(x)-u(y)\right|^p z(x)}{d^{sp}(x,y)\left|B\left(a, \frac{d(x,y)}{2}\right)\right|}\,dx\right)^{\frac{q}{2}} v(y)\, dy\right)^{\frac{2}{q}}\left\|\frac{A^{\frac{1}{q}}u}{|\cdot|_a}\right\|^{\frac{2(q-2)}{q}}_{L^1(\mathbb{X})},
\end{equation*}
    where C is a positive constant independent of $u$.
\end{theorem}
\begin{proof}
    Let us denote $g(x):=A^{\frac{1}{q}}(x)u(x)|x|_a^{-s}$. By using Jensen inequality for the convex function $h(u)=\log\left(\frac{1}{u}\right)$, we have
    \begin{equation}\label{Jensen}
    \begin{split}
        \log\left(\frac{\left\|g\right\|^2_{L^2(\mathbb{X})}}{\left\|g\right\|_{L^1(\mathbb{X})}}\right)&=h\left(\frac{\left\|g\right\|_{L^1(\mathbb{X})}}{\left\|g\right\|^2_{L^2(\mathbb{X})}}\right)=h\left(\int_{\mathbb{X}}\frac{1}{|g(x)|}|g(x)|^2\left\|g\right\|^{-2}_{L^2(\mathbb{X})}\, dx\right)\\
        &\leq\int_{\mathbb{X}}h\left(\frac{1}{|g(x)|}\right)|g(x)|^2\left\|g\right\|^{-2}_{L^2(\mathbb{X})}\, dx=\int_{\mathbb{X}}\frac{|g(x)|^2}{\left\|g\right\|^2_{L^2(\mathbb{X})}}\log|g(x)|\, dx.
        \end{split}
    \end{equation}
    Then by the logarithmic Hardy-Sobolev inequality \eqref{logHardySovolev} for $p=2$, we get
    \begin{equation}\label{logatithmicineq}
        \begin{split}
        &\int_{\mathbb{X}}\frac{|g(x)|^2}{\left\|g\right\|^2_{L^2(\mathbb{X})}}\log|g(x)|\, dx=
            \frac{1}{2}\int_{\mathbb{X}}\frac{|g(x)|^2}{\left\|g\right\|^2_{L^2(\mathbb{X})}}\log\left(\frac{|g(x)|^2}{\left\|g\right\|^2_{L^2(\mathbb{X})}}\left\|g\right\|^2_{L^2(\mathbb{X})}\right)\, dx\\
            &=\frac{1}{2}\int_{\mathbb{X}}\frac{|g(x)|^2}{\left\|g\right\|^2_{L^2(\mathbb{X})}}\log\left(\frac{|g(x)|^2}{\left\|g\right\|^2_{L^2(\mathbb{X})}}\right)\,dx+\int_{\mathbb{X}}\frac{|g(x)|^2}{\left\|g\right\|^2_{L^2(\mathbb{X})}}\log\left\|g\right\|_{L^2(\mathbb{X})}\, dx\\
            &\leq\frac{q}{2(q-2)}\log\left(C\frac{\left(\int_{\mathbb{X}}\left(\int_{\mathbb{X}}\frac{\left|u(x)-u(y)\right|^p z(x)}{d^{sp}(x,y)\left|B\left(a, \frac{d(x,y)}{2}\right)\right|}\,dx\right)^{\frac{q}{2}} v(y)\, dy\right)^{\frac{2}{q}}}{\left\|g\right\|^2_{L^2(\mathbb{X})}}\right)+\log\left\|g\right\|_{L^2(\mathbb{X})}\\
            &=\frac{q}{2(q-2)}\log\left(C\left(\int_{\mathbb{X}}\left(\int_{\mathbb{X}}\frac{\left|u(x)-u(y)\right|^p z(x)}{d^{sp}(x,y)\left|B\left(a, \frac{d(x,y)}{2}\right)\right|}\,dx\right)^{\frac{q}{2}} v(y)\, dy\right)^{\frac{2}{q}}\right)\\
            &-\frac{q}{q-2}\|g\|_{L^2(\mathbb{X})}+\log\left\|g\right\|_{L^2(\mathbb{X})}.
        \end{split}
    \end{equation}
    By using \eqref{Jensen} and \eqref{logatithmicineq}, we have
    \begin{equation}
        \begin{split}
            &2\log\left\|g\right\|_{L^2(\mathbb{X})}-\log\left\|g\right\|_{L^1(\mathbb{X})}-\log\left\|g\right\|_{L^2(\mathbb{X})}+\frac{q}{q-2}\log\left\|g\right\|_{L^2(\mathbb{X})}\\
            &\leq\frac{q}{2(q-2)}\log\left(C\left(\int_{\mathbb{X}}\left(\int_{\mathbb{X}}\frac{\left|u(x)-u(y)\right|^p z(x)}{d^{sp}(x,y)\left|B\left(a, \frac{d(x,y)}{2}\right)\right|}\,dx\right)^{\frac{q}{2}} v(y)\, dy\right)^{\frac{2}{q}}\right).
        \end{split}
    \end{equation}
    Therefore, we obtain
    \begin{equation*}
    \begin{split}
         &\left(1+\frac{q}{q-2}\right)\log\left\|g\right\|_{L^2(\mathbb{X})}\\
         &\leq\frac{q}{2(q-2)}\log\left(C\left\|g\right\|^{\frac{2(q-2)}{q}}_{L^1(\mathbb{X})}\left(\int_{\mathbb{X}}\left(\int_{\mathbb{X}}\frac{\left|u(x)-u(y)\right|^p z(x)}{d^{sp}(x,y)\left|B\left(a, \frac{d(x,y)}{2}\right)\right|}\,dx\right)^{\frac{q}{2}} v(y)\, dy\right)^{\frac{2}{q}}\right),
    \end{split}
    \end{equation*}
    and hence
    \begin{equation*}
            \left\|g\right\|^{4-\frac{4}{q}}_{L^2(\mathbb{X})}\leq\left(C\left\|g\right\|^{\frac{2(q-2)}{q}}_{L^1(\mathbb{X})}\left(\int_{\mathbb{X}}\left(\int_{\mathbb{X}}\frac{\left|u(x)-u(y)\right|^p z(x)}{d^{sp}(x,y)\left|B\left(a, \frac{d(x,y)}{2}\right)\right|}\,dx\right)^{\frac{q}{2}} v(y)\, dy\right)^{\frac{2}{q}}\right).
    \end{equation*}
    We can write this in the following form
    \begin{equation*}
          \left\|\frac{A^{\frac{1}{q}}u}{|\cdot|_a}\right\|^{4-\frac{4}{q}}_{L^2(\mathbb{X})}\leq C \left(\int_{\mathbb{X}}\left(\int_{\mathbb{X}}\frac{\left|u(x)-u(y)\right|^p z(x)}{d^{sp}(x,y)\left|B\left(a, \frac{d(x,y)}{2}\right)\right|}\,dx\right)^{\frac{q}{2}} v(y)\, dy\right)^{\frac{2}{q}}\left\|\frac{A^{\frac{1}{q}}u}{|\cdot|_a}\right\|^{\frac{2(q-2)}{q}}_{L^1(\mathbb{X})},
    \end{equation*}
    which completes the proof.
\end{proof}
\section{Applications}In this section, we show some consequences of the fractional Hardy and fractional Hardy-Sobolev type inequalities on homogeneous Lie groups and on the Heisenberg group.
\subsection{Homogeneous groups.} A Lie group (on $\mathbb{R}^n$) $\mathbb{G}$ is called homogeneous Lie group, if there is a dilation $D_{\lambda}(x)$ such that 
\begin{equation*}
    D_{\lambda}(x):=(\lambda^{\nu_1}x_1, \dots, \lambda^{\nu_n}x_n), \,\, \nu_1,\dots,\nu_n>0, \,\, D_\lambda:\mathbb{R}^n\rightarrow \mathbb{R}^n,
\end{equation*}
and $D_{\lambda}$ is an automorphism of the group $\mathbb{G}$ for all $\lambda>0$. The homogeneous dimension of the homogeneous group $\mathbb{G}$ is denoted by $Q:=\nu_1+\dots+\nu_n$. We also denote a homogeneous quasi-norm on $\mathbb{G}$ by $|x|$, which is a continuous non-negative function
\begin{equation*}
    \mathbb{G}\ni x\mapsto |x|\in[0, \infty),
\end{equation*}
which has the following properties
\begin{enumerate}[label=\Roman*)]
    \item $|x|=|x^{-1}|$ for  all $x\in\mathbb{G}$, 
    \item $|\lambda x|=\lambda|x|$, for all $x\in\mathbb{G}$, and $\lambda>0$,
    \item $|x|=0$ if and only if $x=0$.
\end{enumerate}
\begin{prop}[see e.g. \cite{Folland}, \cite{Fischer}, Proposition 3.1.39 and \cite{RuzhanskySuragan}, Proposition 1.2.4] Let $\mathbb{G}$ be a homogeneous group. Then there always exists a homogeneous quasi-norm $|\cdot|$ on $\mathbb{G}$ which satisfies the triangle inequality (with constant C = 1)
\begin{equation*}
    |xy|\leq|x|+|y|
\end{equation*}
    for all $x, y\in\mathbb{G}$.  
\end{prop}

There is a (unique) positive Borel measure $\sigma$ on the unit quasi-sphere $\mathfrak{S}:=\{x\in\mathbb{G}:|x|=1\}$, so that for every $f\in L^1(\mathbb{G})$, we have 
\begin{equation}\label{polar decomposition of homogenuous groups}
   \int_{\mathbb{G}}f(x)\,dx=\int_0^{\infty}\int_{\mathfrak{S}}f(ry)r^{Q-1}\,d\sigma dr.
\end{equation}
Let us define the quasi-ball centered at $x$ with radius $r$ in the following form
\begin{equation*}
    B(x,r)=\{y\in\mathbb{G}: |x^{-1}y|<r\}. 
\end{equation*}

Let $p>1$ and $s\in(0,1)$. For a measurable function $u:\mathbb{G}\rightarrow \mathbb{R}$, the Gagliardo seminorm is given by
\begin{equation*}
    [u]_{s,p}:=\left(\int_{\mathbb{G}}\int_{\mathbb{G}}\frac{|u(x)-u(y)|^p}{|y^{-1}x|^{sp+Q}}\,dxdy\right)^{\frac{1}{p}}. 
\end{equation*}
The fractional Sobolev space $W^{s,p}(\mathbb{G})$ on the homogeneous Lie group $\mathbb{G}$ is given by
\begin{equation*}
    W^{s,p}(\mathbb{G}):=\left\{u\in L^p(\mathbb{G}): [u]_{s,p}<\infty\right\}, 
\end{equation*}
endowed with the norm 
\begin{equation*}
    \|u\|_{W^{s,p}(\mathbb{G})}:=\|u\|_{L^p(\mathbb{G})}+[u]_{s,p}, \quad u\in W^{s,p}(\mathbb{G}).
\end{equation*}
\begin{theorem}[Fractional Hardy inequality]\label{FracHardyonHomogeneousGroups}
    Let $\mathbb{G}$ be a homogeneous Lie group of homogeneous dimension $Q$, and let $|\cdot|$ be a homogeneous quasi-norm on $\mathbb{G}$. Suppose that $s>0$, $p>1$, and $Q<sp$. Then for all $u\in W^{s,p}(\mathbb{G})$, we have the following inequality
\begin{equation}\label{fractionalHardyonHomogeneuosGroups}
    \int_{\mathbb{G}}\frac{|u(x)|^p}{|x|^{sp}}\, dx\leq \frac{2^{sp+Q}Q(sp+Qp-Q)^p}{|\mathfrak{S}|(sp-Q)^p} [u]_{s,p}^{p}.
\end{equation}
\end{theorem}
\begin{proof}
    First, let us verify the condition \eqref{2} for $\mathbb{X}=\mathbb{G}$, $a=0$ (identity element of $\mathbb{G}$), and $d(x,y)=|y^{-1}x|$ with $A(x)=v(x,y)=1$. By noting that $|B(0, |x|)|=\frac{|\mathfrak{S}||x|^Q}{Q}$ (cf. \eqref{polar decomposition of homogenuous groups}), we obtain
\begin{equation*}
    \begin{split}
        \int_{\mathbb{G}\backslash B(0,|x|)}&|y|^{-sp}|B(0,|y|)|^{-p}\,dy\stackrel{\eqref{polar decomposition of homogenuous groups}}{=}\frac{|\mathfrak{S}|^{-p}}{Q^{-p}}\int_{|x|}^{+\infty}\int_{\mathfrak{S}}r^{-sp-Qp+Q-1}\,d\sigma(y)dr\\
        &\stackrel{-sp-Qp+Q<0}{=}\frac{|\mathfrak{S}|^{1-p}Q^{p}}{(sp+Qp-Q)}|x|^{-sp-Qp+Q}, 
    \end{split}
\end{equation*}
and
\begin{equation*}
        \int_{B(0,|x|)}|y|^{-sp(1-p')}\,dy\stackrel{\eqref{polar decomposition of homogenuous groups}}{=}\int_{0}^{|x|}\int_{\mathfrak{S}}r^{sp'+Q-1}\,d\sigma(y)dr\stackrel{sp'+Q>0}{=}\frac{|\mathfrak{S}|}{sp'+Q}|x|^{sp'+Q}. 
\end{equation*}
Therefore
\begin{equation*}
\begin{split}
    D_1&=\sup_{x\neq 0}\left[\frac{Q|\mathfrak{S}|^{\frac{1-p}{p}+\frac{1}{p'}}}{(sp+Qp-Q)^{\frac{1}{p}}(sp'+Q)^{\frac{1}{p'}}}|x|^{\frac{-sp-Qp+Q}{p}+\frac{sp'+Q}{p'}}\right]\\
    &=\frac{Q}{(sp+Qp-Q)^{\frac{1}{p}}\left(s\frac{p}{p-1}+Q\right)^{\frac{1}{p'}}}=\frac{Q(p-1)^{\frac{1}{p'}}}{sp+Qp-Q}<\infty.
\end{split}
\end{equation*}
By using $Q<sp$, we have
\begin{equation*}
    (p')^{\frac{1}{p'}}p^{\frac{1}{p}}D_1=\frac{Qp}{sp+Qp-Q}<\frac{Qp}{Qp}=1.
\end{equation*} 
By applying the inequality \eqref{Hardy}, and using $|B(0, |x|)|=\frac{|\mathfrak{S}||x|^Q}{Q}$, we get 
\begin{equation*}
\begin{split}
    \int_{\mathbb{G}}\frac{|u(x)|^p}{|x|^{sp}}\, dx&\leq\frac{2^{sp}}{\left(1-\frac{Qp}{sp+Qp-Q}\right)^p}\int_{\mathbb{G}}\int_{\mathbb{G}}\frac{|u(x)-u(y)|^p}{|y^{-1}x|^{sp}|B(0, \frac{|y^{-1}x|}{2})|}\,dxdy\\
    &=\frac{2^{sp+Q}Q(sp+Qp-Q)^p}{|\mathfrak{S}|(sp-Q)^p}\int_{\mathbb{G}}\int_{\mathbb{G}}\frac{|u(x)-u(y)|^p}{|y^{-1}x|^{sp+Q}}\,dxdy, 
    \end{split}
\end{equation*}
completing the proof. 
\end{proof}
\begin{rem}
    In the Abelian (Euclidean) case $\mathbb{G}=(\mathbb{R}^n, +)$, $Q=n$, we have the fractional Hardy inequality under the condition $n<sp$. 
\end{rem}
\begin{theorem}[Fractional Hardy-Sobolev type inequality] Let $\mathbb{G}$ be a homogeneous Lie group of homogeneous dimension $Q$, and let $|\cdot|$ be a homogeneous quasi-norm on $\mathbb{G}$. Assume that $s>0$, $1<p,q<\infty$, and $Q<sp$. Then the following inequality 
\begin{equation}\label{fractionalHadrySobolevinequalityonHomogeneousGroups}
    \int_{\mathbb{G}}\frac{|u(x)|^q}{|x|^{sq}}\, dx\leq \frac{2^{sq+\frac{Qq}{p}}Q^{\frac{q}{p}}(sq+Qq-Q)^q}{|\mathfrak{S}|^{\frac{q}{p}}(sq-Q)^q}\int_{\mathbb{G}}\left(\int_{\mathbb{G}}\frac{|u(x)-u(y)|^p}{|y^{-1}x|^{sp+Q}}\, dx\right)^{\frac{q}{p}}\,dy
\end{equation}
holds for all $u\in W^{s,p}(\mathbb{G})$.
\end{theorem}
\begin{proof}
   The proof of this theorem follows the similar approach to the one used in proving Theorem \ref{FracHardyonHomogeneousGroups}.
\end{proof}
\begin{rem}
    The inequality \eqref{fractionalHadrySobolevinequalityonHomogeneousGroups}, as a special case, implies the Abelian (Euclidean) case $\mathbb{G}=(\mathbb{R}^n, +)$, $Q=n$, under the condition $n<sq$.
\end{rem}
\subsection{Heisenberg group.}
The Heisenberg group $\mathbb{H}^n$ is defined as
\begin{equation*}
    \mathbb{H}^n:=\{\xi=(x, y, t): x\in\mathbb{R}^n, y\in\mathbb{R}^n, t\in\mathbb{R}\}. 
\end{equation*}
Composition law in this group is given by
\begin{equation*}
    \xi\circ\xi'=\left(x+x', y+y', t+t'+\frac{1}{2}(xy'-x'y)\right), 
\end{equation*}
where $\xi'=(x', y', t')\in\mathbb{H}^n$, and $xy$ denotes their standard Euclidean inner product. The inverse element in $\mathbb{H}^n$ is $\xi^{-1}=-\xi$. The homogeneous dimension of the Heisenberg group is $Q=2n+2$. 

For $\xi=(x, y, t)\in\mathbb{H}^n$, define the Koranyi-Folland distance
\begin{equation*}
    d(\xi, 0):=((|x|^2+|y|^2)^2+|t|^2)^{\frac{1}{4}}. 
\end{equation*}
For simplicity, we denote $d(\xi, 0)=d(\xi)$. 
\begin{prop}[\cite{Bonfiglioli}, Proposition 5.14.1]\label{tringle ineq of Koranyi-Folland measure} 
    Let $\mathbb{H}^n$ be a Heisenberg group, and let $d(\cdot)$ be the Koranyi-Folland distance. Then there exists a constant $\beta\geq1$ such that
    \begin{equation*}
        d(\xi\circ\xi')\leq\beta d(\xi)+d(\xi'), \quad \xi, \xi'\in\mathbb{H}^n.
    \end{equation*}
\end{prop}
Let us define the ball $B(\xi, r)$ with center $\xi$ and radius $r$ as follows
\begin{equation*}
    B(\xi, r):=\{\xi'\in\mathbb{H}^n: d(\xi^{-1}\circ\xi')<r\}.
\end{equation*}
Also, $\omega_n$ denotes the Lebesgue measure of $B(0, 1)$, i.e. $\omega_n:=|B(0, 1)|$. 
\begin{prop}[\cite{Bonfiglioli}, Proposition 5.4.4]
    Let $\mathbb{H}^n$ be a Heisenberg group of homogeneous dimension $Q$, and $d(\cdot)$ be the Koranyi-Folland distance. Let $f$ be a function defined on $B(0, r)$. If $f\in L^1(B(0,r))$, we have
    \begin{equation}\label{polardeconHeisenberg}
        \int_{B(0, r)}f(d(\xi))\,d\xi=Q\omega_n\int_0^rs^{Q-1}f(s)\,ds. 
        \end{equation}
\end{prop}

For a measurable function $u:\mathbb{H}^{n}\rightarrow\mathbb{R}$, the Gagliardo seminorm is defined by
\begin{equation*}
    [u]_{s,p}:=\left(\int_{\mathbb{H}^n}\int_{\mathbb{H}^n}\frac{|u(\xi)-u(\xi')|^p}{d(\xi^{-1}\circ \xi')^{sp+Q}}\, d\xi' d\xi\right)^{\frac{1}{p}}. 
\end{equation*}
For $p>1$, and $s\in(0,1)$, the fractional Sobolev space on $\mathbb{H}^n$ is
\begin{equation*}
     W^{s,p}(\mathbb{H}^{n}):=\left\{u\in L^p(\mathbb{H}^n): [u]_{s,p}<\infty\right\},
\end{equation*}
equipped with the norm
\begin{equation*}
     \|u\|_{W^{s,p}(\mathbb{H}^n)}:=\|u\|_{L^p(\mathbb{H}^n)}+[u]_{s,p}, \quad u\in W^{s,p}(\mathbb{H}^n).
\end{equation*}
\begin{theorem}[Fractional Hardy inequality]
    Let $\mathbb{H}^n$ be a Heisenberg group of homogeneous dimension $Q$ and let $d(\cdot)$ be the Koranyi-Folland distance on $\mathbb{H}^n$. Suppose that $s>0$, $p>1$, and $Q<sp$. Then for all $u\in W^{s,p}(\mathbb{H}^{n})$, we have the following inequality
    \begin{equation}
    \int_{\mathbb{H}^n}\frac{|u(\xi)|^p}{d(\xi)^{sp}}\, d\xi\leq \frac{(\beta+1)^{sp+Q}(sp+Qp-Q)^p}{\omega_{n}(sp-Q)^p} [u]_{s,p}^p.
\end{equation}
\begin{proof}
    Similarly to the proof of Theorem \ref{frac Hardy ineq}, by using the Minokwski inequality, we have
    \begin{equation}\label{I1+I2}
        \begin{split}
             \left(\int_{\mathbb{H}^n}\frac{|u(\xi)|^p}{d(\xi)^{sp}}\, d\xi\right)^{\frac{1}{p}}&=\left(\int_{\mathbb{H}^n}\frac{1}{d(\xi)^{sp}}\left|u(\xi)-\frac{1}{|B(0, d(\xi)|}\int_{B(0, d(\xi))}u(\xi')\,d\xi'\right.\right.\\
             &+\left.\left.\frac{1}{B(0, d(\xi))}\int_{B(0, d(\xi))}u(\xi')\,d\xi'\right|^p\,d\xi\right)^{\frac{1}{p}}\\
             &\leq\left(\int_{\mathbb{H}^n}\frac{1}{d(\xi)^{sp}}\left|u(\xi)-\frac{1}{|B(0, d(\xi))|}\int_{B(0, d(\xi))}u(\xi')\,d\xi'\right|^p\,d\xi\right)^{\frac{1}{p}}\\
             &+\left(\int_{\mathbb{H}^n}\frac{1}{d(\xi)^{sp}|B(0, d(\xi))|^p}\left|\int_{B(0,d(\xi))}u(\xi')\,d\xi'\right|^p\,d\xi\right)^{\frac{1}{p}}=I_1+I_2, 
        \end{split}
    \end{equation}
    where
    \begin{equation*}
        I_1=\left(\int_{\mathbb{H}^n}\frac{1}{d(\xi)^{sp}}\left|u(\xi)-\frac{1}{|B(0, d(\xi))|}\int_{B(0, d(\xi))}u(\xi')\, d\xi'\right|^p\, d\xi\right)^{\frac{1}{p}}, 
    \end{equation*}
    and 
    \begin{equation*}
        I_2=\left(\int_{\mathbb{H}^n}\frac{1}{d(\xi)^{sp}|B(0, d(\xi))|^p}\left|\int_{B(0,d(\xi))}u(\xi')\, d\xi'\right|^p\, d\xi\right)^{\frac{1}{p}}.
    \end{equation*}
    By using Proposition \ref{tringle ineq of Koranyi-Folland measure} together with the property $d(\xi^{-1})=d(\xi)$, we obtain 
    \begin{equation}\label{trialngleineq}
        d(\xi^{-1}\circ\xi')\leq\beta d(\xi)+d(\xi')\leq\beta d(\xi)+d(\xi)=(\beta+1)d(\xi),
    \end{equation}
    for all $\xi'\in B(0, d(\xi))$, where $\beta\geq1$. Since $|B(0,r_1)|\leq|B(0,r_2)|$ for $0<r_1\leq r_2$, we have
    \begin{equation}\label{measureofbollsonHeisenberg}
        \left|B\left(0, \frac{d(\xi^{-1}\circ\xi')}{\beta+1}\right)\right|\leq|B(0, d(\xi))|. 
    \end{equation}
    By Hölder’s inequality, we have
    \begin{equation}\label{HolderonHisenberg}
        \left(\int_{B(0,d(\xi))}\left|u(\xi)-u(\xi')\right|\,d\xi'\right)^p\leq\left(\int_{B(0, d(\xi))}|u(\xi)-u(\xi')|^p\,d\xi'\right)\left|B(0,d(\xi))\right|^{p-1}.
    \end{equation}
    Therefore, we obtain
    \begin{equation*}
        \begin{split}
         I_1&=\left(\int_{\mathbb{H}^n}\frac{1}{d(\xi)^{sp}|B(0, d(\xi))|^p}\left|u(\xi)|B(0, d(\xi))|-\int_{B(0, d(\xi))}u(\xi')\, d\xi'\right|^p\, d\xi\right)^{\frac{1}{p}}\\
         &=\left(\int_{\mathbb{H}^n}\frac{1}{d(\xi)^{sp}|B(0, d(\xi))|^p}\left|\int_{B(0, d(\xi))}(u(\xi)-u(\xi'))\, d\xi'\right|^p\, d\xi\right)^{\frac{1}{p}}\\
         &\stackrel{\eqref{HolderonHisenberg}}{\leq}\left(\int_{\mathbb{H}^n}\frac{1}{d(\xi)^{sp}|B(0, d(\xi))|}\int_{B(0, d(\xi))}|u(\xi)-u(\xi')|^p\, d\xi'\, d\xi\right)^{\frac{1}{p}}\\
         &\stackrel{\eqref{trialngleineq}, \eqref{measureofbollsonHeisenberg}}{\leq}(\beta+1)^s\left(\int_{\mathbb{H}^n}\int_{B(0, d(\xi))}\frac{|u(\xi)-u(\xi')|^p}{d(\xi^{-1}\circ\xi')^{sp}|B(0, \frac{d(\xi^{-1}\circ\xi')}{\beta+1})|}\, d\xi'\, d\xi\right)^{\frac{1}{p}}\\
         &\leq(\beta+1)^s\left(\int_{\mathbb{H}^n}\int_{\mathbb{H}^n}\frac{|u(\xi)-u(\xi')|^p}{d(\xi^{-1}\circ\xi')^{sp}|B(0, \frac{d(\xi^{-1}\circ\xi')}{\beta+1})|}\, d\xi'\, d\xi\right)^{\frac{1}{p}}. 
        \end{split}
    \end{equation*}
    By noting that $|B(0, r)|=\omega_{n}\, r^Q$ (cf. \eqref{polardeconHeisenberg}), we get the following estimate
    \begin{equation}\label{I1}
        I_1\leq\frac{(\beta+1)^{s+\frac{Q}{p}}}{\omega_{n}^{\frac{1}{p}}}\left(\int_{\mathbb{H}^n}\int_{\mathbb{H}^n}\frac{|u(\xi)-u(\xi')|^p}{d(\xi^{-1}\circ\xi')^{sp+Q}}\, d\xi'\, d\xi\right)^{\frac{1}{p}}.
    \end{equation}
    Now, let us check the condition \eqref{Condition of the int Hardy ineq} for the functions $g(\xi)=\frac{1}{d(\xi)^{sp}|B(0, d(\xi))|^p}$ and $h(\xi)=\frac{1}{d(\xi)^{sp}}$. We have
    \begin{equation*}
    \begin{split}
        &\int_{\mathbb{H}^n\backslash B(0, d(\xi))}d(\xi')^{-sp}|B(0, d(\xi'))|^{-p}d\xi'=\omega_n^{-p}\int_{d(\xi')\geq d(\xi)}d(\xi')^{-sp-Qp}\,d\xi'\\
        &= \omega_n^{-p}\sum_{k=1}^{+\infty}\int_{2^{k-1}d(\xi)\leq d(\xi')\leq 2^k d(\xi)}d(\xi')^{-sp-Qp}\,d\xi'\stackrel{\eqref{polardeconHeisenberg}}{=}Q\omega_n^{1-p}\sum_{k=1}^{+\infty}\int_{2^{k-1}d(\xi)}^{2^kd(\xi)}r^{-sp-Qp+Q-1}\,dr\\
        &=\frac{Q\omega_n^{1-p}}{-sp-Qp+Q}\left(1-\frac{1}{2^{-sp-Qp+Q}}\right)d(\xi)^{-sp-Qp+Q}\sum_{k=1}^{+\infty}2^{k(-sp-Qp+Q)}\\
        &=\frac{Q\omega_n^{1-p}}{sp+Qp-Q}d(\xi)^{-sp-Qp+Q},
        \end{split}
    \end{equation*}
    \begin{equation*}
            \int_{B(0, d(\xi))}d(\xi')^{-sp(1-p')}\,d\xi'\stackrel{\eqref{polardeconHeisenberg}}{=}Q\omega_n\int_0^{d(\xi)}r^{sp'+Q-1}\,dr=\frac{Q\omega_nd(\xi)^{sp'+Q}}{sp'+Q}.
    \end{equation*}
    Then, we can verify that $D_1=\frac{Q(p-1)^{\frac{1}{p'}}}{sp+Qp-Q}<\infty$, and $(p')^{\frac{1}{p'}}p^{\frac{1}{p}}D_1=\frac{Qp}{sp+Qp-Q}<1$. By applying the integral Hardy inequality \eqref{integralHardyinequalityonmetricmeasurespaces}, we have
    \begin{equation*}\label{I2}
        I_2\leq\frac{Qp}{sp+Qp-Q}\left(\int_{\mathbb{H}^n}\frac{|u(\xi)|^p}{d(\xi)^{sp}}\, d\xi\right)^{\frac{1}{p}}. 
    \end{equation*}
    By putting \eqref{I1} and \eqref{I2} to \eqref{I1+I2}, we obtain
    \begin{equation*}
    \left(\int_{\mathbb{H}^n}\frac{|u(\xi)|^p}{d(\xi)^{sp}}\, d\xi\right)^{\frac{1}{p}}\leq \frac{(\beta+1)^{s+\frac{Q}{p}}(sp+Qp-Q)}{\omega_{n}^{\frac{1}{p}}(sp-Q)} \left(\int_{\mathbb{H}^n}\int_{\mathbb{H}^n}\frac{|u(\xi)-u(\xi')|^p}{d(\xi^{-1}\circ \xi')^{sp+Q}}\, d\xi' d\xi\right)^{\frac{1}{p}},
    \end{equation*}
    which completes the proof. 
\end{proof}
\begin{rem}
    The fractional Hardy inequality on the Heisenberg group for $Q>sp$ was proved in \cite{Adimurthi}. Also, if $p=2$, the best constant was obtained in \cite{Roncal}. 
\end{rem}
\end{theorem}
\begin{theorem}[Fractional Hardy-Sobolev type inequality]
    Let $\mathbb{H}^n$ be a Heisenberg group of homogeneous dimension $Q$ and let $d(\cdot)$ be the Koranyi-Folland distance on $\mathbb{H}^n$. Suppose that $s>0$, $1<p,q<\infty$, and $Q<sp$. Then the following inequality holds for all all $u\in W^{s,p}(\mathbb{H}^{n})$
    \begin{equation*}
    \int_{\mathbb{H}^n}\frac{|u(\xi)|^q}{d(\xi)^{sq}}\, d\xi \leq \frac{(\beta+1)^{sq+\frac{Qq}{p}}(sq+Qq-Q)^q}{\omega_{n}^{\frac{q}{p}}(sq-Q)^q}  \int_{\mathbb{H}^n}\left(\int_{\mathbb{H}^n}\frac{|u(\xi)-u(\xi')|^p}{d(\xi^{-1}\circ \xi')^{sp+Q}}\, d\xi'\right)^{\frac{q}{p}} d\xi.
    \end{equation*}
\end{theorem}
\begin{proof}
    The proof of this theorem is similar to the proof of the previous one. Here, we follow the techniques used in Theorem \ref{theo 3.2} instead of Theorem \ref{frac Hardy ineq}. 
 \end{proof}

\end{document}